\documentclass[10pt]{amsart}
\usepackage{a4}
\usepackage{amssymb}
\usepackage{array}
\usepackage{booktabs}
\usepackage{multirow}
\usepackage{hhline}
\usepackage{comment}
\usepackage{color}

\newtheorem{thm}{Theorem}
\newtheorem{con}{Conjecture}
\newtheorem{quest}{Question}
\newtheorem{cor}[thm]{Corollary}
\newtheorem{lem}[thm]{Lemma}
\newtheorem{prop}[thm]{Proposition}

\numberwithin{thm}{section}
\numberwithin{equation}{section}

\newcommand{\norm}[1]{\left\Vert#1\right\Vert}
\newcommand{\abs}[1]{\left\vert#1\right\vert}

\newcommand{\la}{\langle}
\newcommand{\ra}{\rangle}

\newcommand{\A}{\mathcal{A}}

\newcommand{\Comp}{\mathbb{C}}

\newcommand{\Gb}{\mathbb{G}}

\newcommand{\I}{\mathcal{I}}

\newcommand{\prt}{\widehat{\otimes}}

\begin{document}
\title{Operator biflatness of the $L^1$-algebras of compact quantum groups}
\author{Martijn Caspers}
\author{Hun Hee Lee}
\author{\'Eric Ricard}

\address{M. Caspers, Universit\'e de Franche-Comt\'e, 16 Route de Gray, 25030 Besan\c con, France}
\email{martijn.caspers@univ-fcomte.fr}

\address{Hun Hee Lee,
Department of Mathematical Sciences and Research Institute of Mathematics,
Seoul National University, San56-1 Shinrim-dong Kwanak-gu, Seoul 151-747, Republic of Korea}
\email{lee.hunhee@gmail.com}
\address{\'Eric Ricard, Laboratoire de
Math\'ematiques Nicolas Oresme, Universit\'e de Caen
Basse-Normandie,14032 Caen Cedex, France}
\email{eric.ricard@unicaen.fr}

\keywords{operator biflatness, operator biprojectivity, compact quantum group, $SU_q(n)$, $q$-deformation}
\thanks{2000 \it{Mathematics Subject Classification}.
\rm{Primary 43A20, Secondary 20G42}. \\
The first named author is supported by the ANR project: ANR-2011-BS01-008-01.}

\begin{abstract}
We prove that the $L^1$-algebra of any non-Kac type compact quantum group is not operator biflat. Since operator amenability implies operator biflatness, this result shows that any co-amenable, non-Kac type compact quantum group gives a counter example to the conjecture that $L^1(\Gb)$ is operator amenable if and only if $\Gb$ is amenable and co-amenable for any locally compact quantum group $\Gb$. The result also implies that the $L^1$-algebra of a locally compact quantum group is operator biprojective if and only if $\Gb$ is compact and of Kac type.
\end{abstract}

\maketitle

\section{Introduction}

The convolution algebra $L^1(G)$ with respect to a locally compact group $G$  is a central object in abstract harmonic analysis. Homological/cohomological properties of $L^1(G)$ reflect important features of the underlying group $G$. For example, $L^1(G)$ is {\it amenable} as a Banach algebra if and only if $G$ is an {\it amenable} group \cite{J2}. Recall that a Banach algebra $\A$ is called amenable if every bounded derivation $D : \A \to X^*$ is inner for any $\A$-bimodule $X$ and a locally compact group $G$ is called amenable if there is a left invariant mean on $L^\infty(G)$. Since the work of B.E. Johnson, it has been conjectured that the dual object of $L^1(G)$, namely the Fourier algebra $A(G)$, would reflect the amenability of $G$ in the same way. But Z.-J. Ruan showed that we actually need to move to the category of operator spaces. More precisely, he showed in \cite{Ruan95} that $A(G)$ is operator amenable (i.e. amenable in the category of completely contractive Banach algebras) if and only if $G$ is amenable if and only if $L^1(G)$ is operator amenable with the natural operator space structure. The work of Ruan has demonstrated that the category of operator spaces would be a natural place to work in when we deal with the $L^1$-algebra of a locally compact quantum group $\Gb$.

It is a natural question whether the results of Johnson and Ruan could be extended to the case of locally compact quantum groups. The case of Kac algebras has some satisfactory partial results. Let $\Gb$ be a compact Kac algebra. Then the $L^1$-algebra $L^1(\Gb)$ is operator amenable if and only if $\widehat{\Gb}$ is amenable \cite[Theorem 4.5]{Ruan96}. Note that most of the results in \cite[Theorem 4.5]{Ruan96} have been extended to the case of compact quantum groups by R. Tomatsu \cite{Tom06} except the operator amenability of the $L^1$-algebra. Actually, nothing is known beyond the above results except the following easy implications (\cite[Proposition 2.3]{Ruan95} and \cite[Theorem 3.2]{BT03}), namely,
	$$\text{\it if $L^1(\Gb)$ is operator amenable, then $\Gb$ is amenable and co-amenable.}$$
Here, the amenability of $\Gb$ implies the existence of a left invariant mean on $L^\infty(\Gb)$ and co-amenability of $\Gb$ implies the existence of a bounded approximate identity in $L^1(\Gb)$. Since the partial success of Ruan it has been conjectured by V. Runde that the above implication is actually an equivalence, namely,  
	\begin{con}\label{conjecture}
	$L^1(\Gb)$ is operator amenable if and only if $\Gb$ is amenable and co-amenable.
	\end{con}

Operator amenability is closely related to homological properties, namely, operator biprojectivity and operator biflatness. Recall that a completely contractive Banach algebra $\A$ is called {\it operator biprojective} if there is a completely bounded $\A$-bimodule map $\rho : \A \to \A \prt \A$ such that $m\circ \rho = id_\A$, where $\prt$ is the projective tensor product of operator spaces and $m: \A \prt \A \to \A$ is the algebra multiplication. $\A$ is called {\it operator biflat} if there is a completely bounded $\A$-bimodule map $\theta : (\A \prt \A)^* \to \A^*$ such that $\theta \circ m^* = id_{\A^*}$. Clearly, operator biprojectivity implies operator biflatness. In the case of $L^1(G)$ and $A(G)$ for a locally compact group $G$ we have a good understanding of operator biprojectivity. Indeed, $L^1(G)$ is operator biprojective if and only if $G$ is compact and $A(G)$ is operator biprojective if and only if $G$ is discrete, which makes a complete dual picture. Based on the group situation O. Y. Aristov examined the quantum case. He proved that if $L^1(\Gb)$ for a locally compact quantum group $\mathbb{G}$ is operator biprojective, then $\Gb$ must be compact  \cite[Theorem 4.7]{Aristov}  and if $\Gb$ is a compact Kac algebra, then $L^1(\Gb)$ is operator biprojective  \cite[Theorem 4.12]{Aristov}, which led to the following question of his.
	\begin{quest}\label{question}
	Is $L^1(\Gb)$ operator biprojective for any compact quantum group $\Gb$?
	\end{quest}
Note that the positive solution of the above question for the Kac algebra case has already been proved by Z.-J. Ruan/G. Xu \cite{RX} and M. Daws \cite{Daws} gave another proof for that.

The situation of operator biflatness is more subtle. Here, we recall the well-known fact  \cite[Theorem 2.4]{RX} that a completely contractive Banach algebra $\A$ is operator amenable if and only if $\A$ is operator biflat and $\A$ has a bounded approximate identity. Since $L^1(G)$ always has a bounded approximate identity we know that $L^1(G)$ is operator biflat if and only if $G$ is amenable. Meanwhile, $A(G)$ is known to be operator biflat when $G$ is a quasi-[SIN] group  \cite[Theorem 2.4]{ARS}, which is a strictly bigger class of groups than the class of amenable groups and the class of discrete groups.

In this paper we would like to address a negative solution to Conjecture \ref{conjecture} and a complete answer to Question \ref{question} by investigating operator biflatness. Here is our main theorem.
	\begin{thm}\label{thm-main1}
	Let $\Gb$ be a compact quantum group. If $L^1(\Gb)$ is operator biflat, then $\Gb$ is of Kac type. 
	\end{thm}
This gives a complete characterization of operator amenability of $L^1$-algebras of compact quantum groups combined with the known results \cite[Theorem 3.2]{BT03} and \cite[Theorem 4.5]{Ruan96}.
	\begin{cor}
	Let $\Gb$ be a compact quantum group. Then, $L^1(\Gb)$ is operator amenable if and only if $\Gb$ is co-amenable and of Kac type.
	\end{cor}
Since there are examples of co-amenable, compact quantum groups of non-Kac type, which are automatically amenable, we can see that Conjecture \ref{conjecture} fails to be true. The examples of such quantum groups include the $q$-deformations $G_q$, $0<q<1$ of simple compact Lie groups $G$ (\cite[Lemma 4.10]{Tom07} and \cite[Corollary 6.2]{Subfactor-Ban}) and $SU_q(2)$, $-1<q\ne 0<1$.

Theorem \ref{thm-main1} also gives the following complete characterization of operator biprojectivity of $L^1$-algebras of locally compact quantum groups, which answers Question \ref{question}.
	\begin{cor}
	Let $\Gb$ be a locally compact quantum group. Then, $L^1(\Gb)$ is operator biprojective if and only if $\Gb$ is compact and of Kac type.
	\end{cor}

This paper is organized as follows. In section \ref{sec-Prelim} we collect basic materials about compact quantum groups, their $L^1$-algebras and homological/cohomological properties of Banach algebras. The proof of Theorem \ref{thm-main1} will be given in section \ref{sec-biflat}.

We will assume that the reader is familiar with a basic operator space theory, for which we refer to \cite{P03, ER}.

\section{Preliminaries} \label{sec-Prelim}

\subsection{Compact quantum groups and their $L^1$-algebras}

We recall some preliminaries on compact quantum groups, c.f. \cite{Wor} and also \cite{Tim}. A compact quantum group $\Gb$ is given by $(A, \Delta)$, a unital $C^*$-algebra $A$ and a unital $*$-homomorphism $\Delta : A \to A \otimes_{\min} A$ which is {\it co-associative} $(i.e.\; (\Delta\otimes id)\Delta = (id \otimes \Delta)\Delta)$ and satisfies the {\it cancellation law}, namely, both of the spaces
	$$\Delta(A)(A\otimes 1) = \text{span}\{\Delta(a)(b\otimes 1) : a,b\in A\}\;\; \text{and}\;\; \Delta(A)(1\otimes A)$$
are dense in $A \otimes_{\min} A$. We sometimes denote $A$ by $C(\Gb)$ and $\Delta$ is called the {\it co-multiplication} of $\Gb$.

For a compact quantum group $\Gb = (A, \Delta)$ there is a unique {\it Haar state} $h$ on $A$ such that
	$$(h\otimes id)\Delta(a) = h(a)1 = (id \otimes h)\Delta(a),\;\; a\in A.$$

The $L^1$-algebra of compact quantum group can be defined as follows. The {\it reduced} version $C(\Gb)_{\text{red}}$ of $C(\Gb)$ is the $C^*$-algebra $\rho(C(\Gb))$ for the GNS representation $\rho$ of $h$. We define $L^\infty(\Gb)$, the $L^\infty$-space over $\Gb$, to be the von Neumann algebra $C(\Gb)_{\text{red}}''$. Then the co-multiplication $\Delta : C(\Gb) \to C(\Gb) \otimes_{\min} C(\Gb)$ can be naturally extended to a unital normal $*$-isomorphism $L^\infty(\Gb) \to L^\infty(\Gb) \bar{\otimes} L^\infty(\Gb)$. By abuse of notation we will still denote the extended co-multiplication by $\Delta$. Thus, there is a completely contractive pre-adjoint map $\Delta_* : L^1(\Gb) \prt L^1(\Gb) \to L^1(\Gb)$, where $L^1(\Gb) = L^\infty(\Gb)_*$. This gives us a completely contractive Banach algebra structure on $L^1(\Gb)$, which will be assumed as the underlying Banach algebra structure whenever we deal with $L^1(\Gb)$. We will use the following convolution notation.
	$$f*g := \Delta_*(f\otimes g) = (f\otimes g) \Delta,\; f,g\in L^1(\Gb).$$
A compact quantum group $\Gb$ is called {\it co-amenable} if $L^1(\mathbb{G})$ has a bounded approximate identity.

A {\it finite dimensional corepresentation} of $\Gb$ is a matrix $u = (u_{ij}) \in M_n(A)$ such that
	$$\Delta(u_{ij}) = \sum^n_{k=1} u_{ik}\otimes u_{kj},\;\; 1\le i,j \le n.$$
We say that the corepresentation $u$ is {\it unitary} if $u$ is a unitary matrix and {\it irreducible} if we have $\{X \in M_n : Xu=uX\} = \mathbb{C} 1$. The number $n$ is called the {\it dimension} of $u$.

Let $\{u^\alpha : \alpha \in \I\}$ be a maximal family of finite dimensional irreducible unitary corepresentations of $\Gb$. For each $\alpha \in \I$ we denote
	$$n_\alpha = \text{dim} \: u^\alpha.$$
It is well known that for each $\alpha\in \I$ there is a unique positive invertible matrix $Q^\alpha \in M_{n_\alpha}$ with $\text{Tr}\: Q^\alpha = \text{Tr}(Q^\alpha)^{-1}$ satisfying the {\it Schur orthogonality relations}, namely for $\alpha, \beta \in \I$ and $1\le i,j \le n_\alpha$, $1\le k,l \le n_\beta$ we have
	$$h((u^\alpha_{ij})^* u^\beta_{kl}) = \delta_{\alpha \beta}\delta_{jl}\frac{((Q^\alpha)^{-1})_{ki}}{\text{Tr}\:Q^\alpha},\; h(u^\alpha_{ij} (u^\beta_{kl})^*) = \delta_{\alpha \beta}\delta_{ik}\frac{(Q^\alpha)_{lj}}{\text{Tr}\:Q^\alpha}.$$
The matrix $Q^\alpha$ gives us the {\it quantum dimension} $m_\alpha$ of $u^\alpha$ given by
	$$m_\alpha = \text{Tr}\: Q^\alpha=\text{Tr}\: (Q^\alpha)^{-1}.$$
Moreover, this $Q^\alpha$ can be chosen to be diagonal  (\cite[Proposition 2.1]{Daws}), so that we set
	$$Q^\alpha = \text{diag}(\lambda^\alpha_1, \cdots, \lambda^\alpha_{n_\alpha}).$$
In case all matrices $Q^\alpha, \alpha \in \mathcal{I}$ are equal to the identity, we say $\Gb$ is of {\it Kac type}, or a {\it Kac algebra}.

There is a more direct connection between $u^\alpha$ and $Q^\alpha$ as follows.
	\begin{prop}\label{prop-interwiner}
	For $\alpha \in \I$ and $1\le i \le n_\alpha$ the matrix $Q^\alpha$ is the unique matrix satisfying
		\begin{equation}\label{eq-interwiner-full}
		(u^\alpha)^t\overline{Q^\alpha}\overline{u^\alpha} = \overline{Q^\alpha}
		\end{equation}
and $\text{\rm Tr}(Q^\alpha) = \text{\rm Tr}((Q^\alpha)^{-1})$. Here, $\overline{u^\alpha}=((u^\alpha_{ij})^*)$.
	\end{prop}
\begin{proof}
By \cite[Proposition 3.2.17]{Tim} we know that $\overline{Q^\alpha}$ is an intertwiner from $\overline{u^\alpha}$ to $((u^\alpha)^t)^{-1}$, where $(u^\alpha)^t = (u^\alpha_{ji})$. Thus, we have
	$$\overline{Q^\alpha}\overline{u^\alpha} = ((u^\alpha)^t)^{-1}\overline{Q^\alpha} \Leftrightarrow (u^\alpha)^t\overline{Q^\alpha}\overline{u^\alpha} = \overline{Q^\alpha}.$$
The uniqueness also comes from \cite[Proposition 3.2.17]{Tim}.
\end{proof}

The above $Q^\alpha$ matrices are closely related to the following family of Woronowicz characters $f_z: {\rm Pol}(\Gb) \rightarrow \mathbb{C}$. Here, Pol$(\Gb)$ is the Hopf $*$-algebra generated by $\{u^\alpha_{ij} : \alpha\in \I,\; 1\le i,j \le n_\alpha \}$. Actually, $\{u^\alpha_{ij} : \alpha\in \I,\; 1\le i,j \le n_\alpha \}$ is a linear basis of Pol$(\Gb)$, and Pol$(\Gb)$ is equipped with the same co-multiplication $\Delta$, the coinverse $\kappa:  {\rm Pol}(\Gb) \rightarrow {\rm Pol}(\Gb)$ and the counit $\varepsilon: {\rm Pol}(\Gb) \rightarrow \mathbb{C}$. The characters $\{f_z\}_{z\in \Comp}$ have the following properties:
	\begin{enumerate}
		\item $f_z * f_{z'} = f_{z+z'}$, $z, z'\in \Comp$, $f_0 = \varepsilon$.
		\item $f_z \kappa (a) = f_{-z}(a)$, $f_z(a^*) = \overline{f_{-\overline{z}}(a)}$, $a\in \text{Pol}(\Gb)$, $z\in \Comp$, where $\kappa$ is the usual coinverse (or antipode) of $\Gb$.
		\item $\kappa^2(a) = f_{-1}*a*f_1$, $a\in \text{Pol}(\Gb)$.
		\item $h(ab) = h(b f_1 *a*f_1)$, $a, b\in \text{Pol}(\Gb)$.
	\end{enumerate}	
Here, $*$ denotes the convolution defined by
\[
\omega \ast \theta = (\omega \otimes \theta)\Delta, \qquad \omega \ast a = (\omega \otimes id) \Delta(a), \qquad a \ast \omega = (id \otimes \omega)\Delta(a),
\] 
where $\omega, \theta: {\rm Pol}(\Gb) \rightarrow \mathbb{C}$, $a \in {\rm Pol}(\mathbb{G})$.

Then, for a finite dimensional representation $u = (u_{ij}) \in M_n(A)$ we define
	$$Q^u := (id \otimes f_1)u.$$
Note that we have $Q^{u^\alpha} = Q^\alpha$ (c.f. \cite[(5.23)]{Wor}).

The following is a slight modification of \cite[Lemma 1.3]{Subfactor-Ban}.
	\begin{lem}\label{lem-fusion}
		For finite dimensional representations $u$ and $v$ of $\Gb$ we have
			\begin{enumerate}
				\item $Q^{u \oplus v} = Q^u \oplus Q^v$.
				\item $Q^{u \otimes v} = Q^u \otimes Q^v$.
			\end{enumerate}
		
	\end{lem}
\begin{proof}
(1) is clear and (2) is from the fact that $f_1$ is a character.
\end{proof}

\subsection{Homological/cohomological properties of Banach algebras}

Let $\A$ be a completely contractive Banach algebra, i.e. $\A$ is a Banach algebra with an operator space structure, and the algebra multiplication map
	$$m: \A \prt \A \to \A$$
extends to a complete contraction, where $\prt$ is the projective tensor product of operator spaces.

We say that an $\A$-bimodule $X$, which is also an operator space, is an operator $\A$-bimodule if there is a completely bounded map $\A \prt X \prt \A \to X, \; a\otimes x \otimes b \mapsto a\cdot x \cdot b$. For any operator $\A$-bimodule the (operator) dual space $X^*$ is also an operator $\A$-bimodule via the following module structure:
	$$\la a\cdot \varphi \cdot b, x \ra := \la \varphi, b \cdot x \cdot a \ra,\; x\in X, \varphi \in X^*, a, b\in \A.$$
For any two operator $\A$-bimodules $X$ and $Y$ their projective tensor product $X\prt Y$ can be regarded as an operator $\A$-bimodule via the following module structure:
	$$a\cdot (x\otimes y) \cdot b := (a\cdot x) \otimes (y\cdot b),\; x\in X, y\in Y, a, b \in A.$$
If we combine the above two concepts, then $(X\prt Y)^*$ becomes an operator $\A$-bimodule via the following module structure:
	$$a\cdot (x^*\otimes y^*) \cdot b := (x^* \cdot b) \otimes (a \cdot y^*),\; x^*\in X^*, y^*\in Y^*, a, b \in A.$$ Note that the above module action looks a little bit twisted.
	
We say that $\A$ is {\it operator biprojective} if there is a completely bounded $\A$-bimodule map $\rho : \A \to \A \prt \A$ such that $m\circ \rho = id_\A$. We say that $\A$ is {\it operator biflat} if there is a completely bounded $\A$-bimodule map $\theta : (\A \prt \A)^* \to \A^*$ such that $\theta \circ m^* = id_{\A^*}.$ In the above cases $\rho$ and $\theta$ are called {\it splitting homomorphisms}. Clearly, operator biprojectivity implies operator biflatness.

We say that $\A$ is {\it operator amenable} if every completely bounded derivation $D : \A \to X^*$ is inner for any operator $\A$-bimodule $X$, i.e. there exists a $f \in X^\ast$ such that $D(a) = a\cdot f - f \cdot a$ for each $a\in \A$.

\section{Proof of the main result}\label{sec-biflat}

In \cite{Daws} M. Daws made the following important observation while he was investigating operator biprojectivity of $L^1(\Gb)$. Recall that we made the assumption that $Q^\alpha$ is a diagonal matrix for every $\alpha \in \mathcal{I}$, which plays a role in the proof of M. Daws.
\begin{prop} The following are equivalent.
	\begin{enumerate}
		\item
		 $L^1(\Gb)$ is operator biprojective with splitting homomorphism $\rho: L^1(\Gb) \to L^1(\Gb) \prt L^1(\Gb)$.
		\item
		There is a normal completely bounded map $\theta : L^\infty(\Gb) \bar{\otimes} L^\infty(\Gb) \to L^\infty(\Gb)$ such that
			\begin{align}
			\theta \circ \Delta & = id, \label{eq1}\\
			\Delta \circ \theta & = (\theta \otimes id)(id \otimes \Delta) = (id \otimes \theta)(\Delta \otimes id). \label{eq2}
			\end{align}
		\item
		There is a normal completely bounded map $\theta : L^\infty(\Gb) \bar{\otimes} L^\infty(\Gb) \to L^\infty(\Gb)$ and a family $\{X^\alpha \in M_{n_\alpha}: \alpha \in \I \}$ such that for $\alpha, \beta \in \I$, $1\le i,j\le n_\alpha$ and $1\le k,l\le n_\beta$,
			\begin{equation} \label{eq3}
			\theta(u^\alpha_{ij} \otimes u^\beta_{kl}) = \delta_{\alpha \beta}X^\alpha_{jk}u^\alpha_{il},\;\; \sum^{n_\alpha}_{r=1}X^\alpha_{rr} = 1.
			\end{equation}
	\end{enumerate}
If one of these conditions is satisfied, then the notation is consistent in the sense that $\theta$ in (2) and (3) can be taken the same and one may take $\rho^* = \theta$.
\end{prop}

Our starting point is that if we assume operator biflatness of $L^1(\Gb)$, the corresponding splitting homomorphism $\theta: L^\infty(\Gb) \bar{\otimes} L^\infty(\Gb) \to L^\infty(\Gb)$ still satisfies the condition \eqref{eq3}.

\begin{prop}\label{prop-splitting-form}
Suppose that $L^1(\Gb)$ is operator biflat with splitting homomorphism $\theta: L^\infty(\Gb) \bar{\otimes} L^\infty(\Gb) \to L^\infty(\Gb)$. Then, $\theta$ satisfies \eqref{eq3}.
\end{prop}
\begin{proof}
First, \eqref{eq1} is trivial from the definition of the splitting homomorphism. Secondly, the condition \eqref{eq2} can be replaced by the $L^1(\Gb)$-bimodule property of $\theta$. Indeed, for $x\in L^\infty(\Gb)$ and $\varphi, \psi \in L^1(\Gb)$ we have 
	\begin{align*}
	\la \Delta \circ \theta (x\otimes u^\alpha_{ij}), \varphi \otimes \psi \ra
	& = \la \theta (x \otimes u^\alpha_{ij}), \varphi*\psi \ra\\
	& = \la \psi * \theta (x\otimes u^\alpha_{ij}), \varphi \ra\\
	& = \la \theta[\psi * (x\otimes u^\alpha_{ij})], \varphi \ra\\
	& = \la \theta[x \otimes \psi * u^\alpha_{ij}], \varphi \ra.
	\end{align*}
Moreover, for any $\phi\in L^1(\Gb)$ we have $\la \psi * u^\alpha_{ij}, \phi \ra = \la u^\alpha_{ij}, \phi * \psi \ra = \la \Delta u^\alpha_{ij}, \phi \otimes \psi \ra$, so that we have
	$$\psi * u^\alpha_{ij} = (id \otimes \psi)\Delta u^\alpha_{ij} = \sum^{n_\alpha}_{r=1} \psi(u^\alpha_{rj})u^\alpha_{ir}.$$
Thus, we have
	$$\la \Delta \circ \theta (x\otimes u^\alpha_{ij}), \varphi \otimes \psi \ra = \sum^{n_\alpha}_{r=1} \psi(u^\alpha_{rj}) \la \theta(x\otimes u^\alpha_{ir}), \varphi \ra = \left\la \sum^{n_\alpha}_{r=1} \theta(x\otimes u^\alpha_{ir}) \otimes u^\alpha_{rj}, \varphi \otimes \psi \right\ra.$$
Consequently, we get
	$$\Delta \circ \theta (x\otimes u^\alpha_{ij}) = \sum^{n_\alpha}_{r=1} \theta(x\otimes u^\alpha_{ir}) \otimes u^\alpha_{rj}.$$
Note that we do not need normality of $\theta$, which is the main difference from the case of \cite[Lemma 3.1, Proposition 3.2]{Daws}. From this point on we can repeat the same argument of the proof of \cite[Proposition 3.2]{Daws} to get the conclusion we wanted.
\end{proof}

Now we present the proof of our main result. Here, we use some operator space theory. Recall that $R_n$ and $C_n$ are the row and column Hilbert space on $\ell^2_n$. For any operator space $E \subset B(H)$ for a Hilbert space $H$ and an element $\sum^n_{i=1} e_i \otimes x_i \in \ell^2_n \otimes E$ we have
	$$\norm{\sum^n_{i=1} e_i \otimes x_i}^2_{R_n \otimes_{\min} E} = \norm{\sum^n_{i=1}x_i x^*_i}_{B(H)},\;\; \norm{\sum^n_{i=1} e_i \otimes x_i}^2_{C_n \otimes_{\min} E} = \norm{\sum^n_{i=1}x^*_i x_i}_{B(H)},$$
where $\otimes_{\min}$ is the injective tensor product of operator spaces. Moreover, we have a completely isometric identification
	$$M_n \cong C_n \otimes_{\min} R_n,\;\; e_{ij}\mapsto e_i \otimes e_j.$$

By Proposition \ref{prop-splitting-form} we can assume $\theta$ satisfies the condition \eqref{eq3}. Now we fix $\alpha \in \I$. Then for $1\le i,j,k,l \le n_\alpha$ we have
	$$\theta(u^\alpha_{ij} \otimes u^\alpha_{kl}) = X^\alpha_{jk}u^\alpha_{il}.$$
From \eqref{eq-interwiner-full}, $[(u^\alpha)^t]^* = \overline{u^\alpha}$ and the fact that $Q^\alpha$ is diagonal we have $(u^\alpha)^tQ^\alpha = Q^\alpha (u^\alpha)^t$, so that we get
	\begin{align*}
	\norm{(Q^\alpha)^{-\frac{1}{2}}(u^\alpha)^t (Q^\alpha)^\frac{1}{2}}_{M_{n_\alpha} (L^\infty(\Gb))} & = \norm{(Q^\alpha)^{-\frac{1}{2}}(u^\alpha)^t Q^\alpha [(u^\alpha)^t]^* (Q^\alpha)^{-\frac{1}{2}}}^{\frac{1}{2}}_{M_{n_\alpha} (L^\infty(\Gb))}\\ & =  \norm{1}_{M_{n_\alpha} (L^\infty(\Gb))} = 1.
	\end{align*}	
Moreover, we have
	$$(Q^\alpha)^{-\frac{1}{2}}(u^\alpha)^t (Q^\alpha)^\frac{1}{2} = \left(u^\alpha_{ji}\sqrt{\frac{\lambda^\alpha_j}{\lambda^\alpha_i}}\right) = \sum^{n_\alpha}_{i,j=1}e_{ij}\otimes u^\alpha_{ji}\sqrt{\frac{\lambda^\alpha_j}{\lambda^\alpha_i}} \in M_{n_\alpha} (L^\infty(\Gb)).$$	
Since $u^\alpha$ is unitary we also have
	\begin{align*}
	\lefteqn{\norm{(Q^\alpha)^{-\frac{1}{2}}(u^\alpha)^t (Q^\alpha)^\frac{1}{2} \otimes u^\alpha}_{M_{n_\alpha} (L^\infty(\Gb)) \bar{\otimes} M_{n_\alpha} (L^\infty(\Gb))}}\\ & = \norm{\sum^{n_\alpha}_{i,j,k,l=1}e_{ij} \otimes e_{kl} \otimes u^\alpha_{ji}\sqrt{\frac{\lambda^\alpha_j}{\lambda^\alpha_i}} \otimes u^\alpha_{kl}}_{M_{n^2_\alpha} (L^\infty(\Gb)\bar{\otimes} L^\infty(\Gb))} = 1.
	\end{align*}
Then, we have
	\begin{align*}
		\lefteqn{(id_{M_{n_\alpha}} \otimes id_{M_{n_\alpha}}
                  \otimes
                  \theta)\left(\sum^{n_\alpha}_{i,j,k,l=1}e_{ij}
                  \otimes e_{kl} \otimes
                  u^\alpha_{ji}\sqrt{\frac{\lambda^\alpha_j}{\lambda^\alpha_i}}
                  \otimes u^\alpha_{kl}\right)}\\ & =
                \sum^{n_\alpha}_{i,j,k,l=1}e_{ij} \otimes e_{kl}
                \otimes
                X^\alpha_{ik}\sqrt{\frac{\lambda^\alpha_j}{\lambda^\alpha_i}}
                u^\alpha_{jl}\\ & = (\tau_{23} \otimes
                id_{L^\infty(\Gb)})\left[\left(
                  \sum^{n_\alpha}_{i,k=1} X^\alpha_{ik}\frac 1
                  {\sqrt{\lambda^\alpha_i}} e_i \otimes e_k \right)
                  \otimes \left( \sum^{n_\alpha}_{j,l=1}e_j \otimes
                  e_l \otimes {\sqrt{\lambda^\alpha_j}} u^\alpha_{jl}
                  \right)\right],
	\end{align*}
where $\tau_{23}$ is the map flipping the second and the third tensor component in
	$$M_{n_\alpha} \otimes_{\min} M_{n_\alpha} \cong C_{n_\alpha}\otimes_{\min} R_{n_\alpha} \otimes_{\min} C_{n_\alpha} \otimes_{\min} R_{n_\alpha}.$$
Now we have
	$$\sum^{n_\alpha}_{i,k=1} \frac {X^\alpha_{ik}}{\sqrt{\lambda^\alpha_i}} e_i \otimes e_k \in C_{n_\alpha}\otimes_{\min} C_{n_\alpha} \cong C_{n^2_\alpha}$$
and
	$$\sum^{n_\alpha}_{j,l=1}e_j \otimes e_l \otimes {\sqrt{\lambda^\alpha_j}} u^\alpha_{jl} \in R_{n_\alpha}\otimes_{\min} R_{n_\alpha} \otimes_{\min} L^\infty(\Gb) \cong R_{n^2_\alpha} \otimes_{\min} L^\infty(\Gb)$$
with the corresponding norms 	
	$$\left(\sum^{n_\alpha}_{i,k=1}\frac{\abs{X^\alpha_{ik}}^2}{ \lambda^\alpha_i}\right)^{\frac{1}{2}}$$
and
	$$\norm{\sum^{n_\alpha}_{j,l=1}\lambda^\alpha_ju^\alpha_{jl}(u^\alpha_{jl})^*}^{\frac{1}{2}}_{L^\infty(\Gb)} = \left(\sum^{n_\alpha}_{j=1}\lambda^\alpha_j\right)^{\frac{1}{2}} = \sqrt{m_\alpha},$$
respectively. This gives us a lower bound of $\norm{\theta}_{cb}$ by
	$$\norm{\theta}_{cb} \ge \left(\sum^{n_\alpha}_{i,k=1}\frac {\abs{X^\alpha_{ik}}^2}{ \lambda^\alpha_i}\right)^{\frac{1}{2}}\sqrt{m_\alpha} \ge \left(\sum^{n_\alpha}_{i=1}
\frac{\abs{X^\alpha_{ii}}^2}{ \lambda^\alpha_i}\right)^{\frac{1}{2}}\sqrt{m_\alpha}.$$
If we repeat the same estimate for $\displaystyle \sum^{n_\alpha}_{i,j,k,l=1}e_{ij} \otimes e_{kl} \otimes u^\alpha_{ij} \otimes u^\alpha_{lk}\sqrt{\frac{\lambda^\alpha_l}{\lambda^\alpha_k}}$, then we get
	$$\norm{\theta}_{cb} \ge \left(\sum^{n_\alpha}_{i=1}\abs{X^\alpha_{ii}}^2 \lambda^\alpha_i\right)^{\frac{1}{2}}\sqrt{m_\alpha}.$$
Combining the above two we have
	$$\norm{\theta}^2_{cb} \ge m_\alpha \sum^{n_\alpha}_{i=1}\abs{X^\alpha_{ii}}^2 \left(\frac{\lambda^\alpha_i+ (\lambda^\alpha_i)^{-1}}{2}\right) \ge m_\alpha \sum^{n_\alpha}_{i=1}\abs{X^\alpha_{ii}}^2 \ge \frac{m_\alpha}{n_\alpha}$$
from the Cauchy-Schwarz inequality and $\displaystyle \sum^{n_\alpha}_{i=1}X^\alpha_{ii} = 1$, so that we have
	$$m_\alpha \le \norm{\theta}^2_{cb} n_\alpha.$$
We can improve this estimate by tensoring. For any $d\in \mathbb{N}$ we consider the tensor power $(u^\alpha)^{\otimes d}$ which will be decomposed into the direct sum $\oplus^N_{k=1}v_k$ of finite dimensional irreducible unitary corepresentations. By applying Lemma \ref{lem-fusion} we have
	\begin{align*}
	(m^\alpha)^d & = (\text{Tr}\:Q^\alpha)^d = \text{Tr}\:Q^{(u^\alpha)^{\otimes d}} = \sum^N_{k=1} \text{Tr}\:Q^{v_k}\\
	& \le \norm{\theta}^2_{cb} \sum^N_{k=1} \text{dim}\:v_k = \norm{\theta}^2_{cb} \text{dim}\:(u^\alpha)^{\otimes d} = \norm{\theta}^2_{cb} (n^\alpha)^d.
	\end{align*}
By taking the power  $\frac{1}{d}$ and letting $d \to \infty$ we get
	$$m_\alpha \le n_\alpha,$$
which implies that $\Gb$ is of Kac type.

{\it Acknowledgement: The authors wish to thank Volker Runde for giving us the details about the problem and Adam Skalski for providing the reference \cite{Tom07}.}

\end{document}